\documentclass[reqno, 10pt]{amsart}
\usepackage{color}

\usepackage{graphicx, enumerate}
\usepackage{amssymb, amsmath, amsthm}
\allowdisplaybreaks

\newcommand{\modulo}{\operatorname{mod}}

\newcommand{\euler}{\epsilon}

\newcommand{\rfac}[2]{{\left({#1}\right)_{#2}}}

\newcommand{\pqrfac}[3]{{\left({#1};#3\right)_{#2}}}

\newcommand{\A}{{ \boldsymbol A}}

\usepackage{hyperref}
\hypersetup{
    colorlinks=true, 
    linktoc=all,     
    linkcolor=blue,  
    citecolor=blue
}
\DeclareMathOperator{\fib}{Fib}

\newtheorem{Theorem}{Theorem}[section]
\newtheorem{Proposition}[Theorem]{Proposition}
\newtheorem{Lemma}[Theorem]{Lemma}

\newtheorem{Corollary}[Theorem]{Corollary} 

\theoremstyle{definition}
\newtheorem{Example}{Example}[section]
\newtheorem{Examples}{Examples}[section]

\theoremstyle{remark}
\newtheorem*{Remarks}{Remarks}
\newtheorem*{Remark}{Remark}

\numberwithin{equation}{section}

\begin{document}

\title{The Partition-Frequency Enumeration Matrix}

\author[H.~S.~Bal]{Hartosh Singh Bal
}
\address{The Caravan\\
Jhandewalan Extn., New Delhi 110001, India}
\email{hartoshbal@gmail.com}

\author[G.~Bhatnagar]{Gaurav Bhatnagar
}
\address{Ashoka University, Sonipat, Haryana, India}
\email{bhatnagarg@gmail.com}
\urladdr{https://www.gbhatnagar.com}

\date{\today}


\begin{abstract}
We develop a calculus that gives an elementary approach to enumerate partition-like objects using an infinite 
number-theoretic matrix. We call this matrix the partition-frequency enumeration (PFE) matrix. This matrix unifies a large number of 
results connecting number-theoretic functions to partition-theoretic functions. The calculus is extended to arbitrary generating functions, and functions with Weierstrass products. As a by-product, we recover (and extend)
some well-known recurrence relations for many number-theoretic functions, including the sum of divisors function, Ramanujan's $\tau$ function, sums of squares and triangular numbers, and for $\zeta(2n)$, where $n$ is a positive integer. These include classical results due to Euler, Ewell, Ramanujan, Lehmer and others.  As one application, 
we 
embed Ramanujan's famous congruences $p(5n+4)\equiv 0\; (\modulo 5)$
and $\tau(5n+5)\equiv 0\; (\modulo 5)$
into an infinite family of such congruences. 
\end{abstract}

\maketitle

\section{Introduction}
 Let $\sigma_1(n)$ be the sum of divisors of $n$. Euler~\cite{LE1760-243} showed that
\begin{equation}\label{euler-sigma}
\sigma_1(n) = \sigma_1(n-1)+\sigma_1(n-2)-\sigma_1(n-5)-\sigma_1(n-7)+\sigma_1(n-12)
+\cdots.
\end{equation} 
There are two striking features of this result. One, this is a recurrence relation for $\sigma_1(n)$,  a multiplicative arithmetic function. And two, the numbers appearing in Euler's recurrence: $1$, $2$, $5$, $7$, $\dots$, are the generalized pentagonal numbers which are related to the partition 
function $p(n)$---the number of ways of writing a positive integer as an unordered sum of positive integers---which is part of additive number theory. The numbers are quite far apart, which makes it convenient to compute $\sigma_1(n)$ for small values of $n$. These numbers feature in Euler's pentagonal number theorem, which is the expansion
$$(1-q)(1-q^2)(1-q^3)\cdots = 1-q-q^2+q^5+q^7-q^{12}-q^{15}+\cdots.$$
As is well-known, the reciprocal of the product appearing in Euler's pentagonal number theorem is a generating function of integer partitions. 

Such recurrences have been found for many number-theoretic functions, including Ramanujan's $\tau$ function and $r_k$, the number of ways of writing a number $n$ as an ordered sum of $k$ squares; in addition, there are many recurrence relations that contain a partition-theoretic function. Aside from Euler, such results have been found by Ramanujan, Lehmer, Ewell; and many are found in number theory texts without attribution.

In this paper, we
 study a number-theoretic matrix that seems to be at the heart of the connection between arithmetic functions and partition functions, and such recurrence relations. This matrix is closely related to Redheffer's matrix (see Vaughan~\cite{Vaughan1993}), but the use we make of it is quite different. 

Consider the matrix 
\begin{equation}\label{partition-matrix}
A=
\begin{pmatrix}
1 & 1 &  1 & 1 & 1 & 1& \hdotsfor 1\\
0 & 1 & 0 & 1  & 0 & 1 &\hdotsfor1 \\
0 & 0 & 1 & 0  & 0 & 1 &  \hdots    \\ 
0 & 0 & 0 & 1  & 0 & \hdotsfor 2 \\ 
\vdots & \vdots  & \vdots & \vdots &\vdots & \vdots &\vdots\\
\end{pmatrix},
\end{equation}
with the $(i,j)$th entry given by $1$ when $i \mid j$, and $0$ otherwise. Redheffer's matrix has all
$1$'s in the first column; the rest of the entries are the same.

Let $A_n$ be the $n\times n $ sub-matrix consisting of the entries from the first $n$ rows and columns of $A$. For $n=1, 2, \dots$, in turn, we generate the vectors
$P_n=(P(n-1),P(n-2),\dots, P(0))$ and $F(n) = (F_1(n),F_2(n),\dots, F_n(n))$ by means of the two equations
\begin{subequations}
\begin{align}
A_n P_n^T &= F(n)^T \label{part-a}\\
\sum_{k=1}^n kF_k(n) &= nP(n) \label{part-b}
\end{align}
\end{subequations}
together with the initial condition
$P(0)=1$. (Here the superscript $T$ is used to denote transpose.)
The first few terms of the sequence $P(n)$  generated in this fashion are
$$
1, 1, 2, 3, 5, 7, 11, 15, 22, 30, 42, 56, 77, 101, 135, 176, 231, 297, 385, 490, 627,\dots.
$$
This is the sequence $p(n)$ of integer partitions.  (See Example~\ref{ex:partitions} for
more details.)


The objective of this paper is to 
utilize this elementary observation to find results that
connect the arithmetic functions of number theory with partition-theoretic functions. We develop a calculus to determine the matrix associated to all functions that can be represented as infinite products or series. As a by-product, we recover many classical identities and recurrence relations, and find analogues  for other interesting functions, and some natural generalizations. 
See Ono, Schneider and Wagner~\cite{OSW2017, OSW2020} and the references cited therein,
and Merca~\cite{Merca2018a, Merca2021a, Merca2022a}, 
for different, possibly complementary, connections between arithmetic and partition-theoretic functions.

Before describing our results, as motivation, we prove our assertion that the sequence obtained from \eqref{part-a} and \eqref{part-b} 
is in fact $p(n)$, the partition function. 

We find it convenient to 
use the symbols $(u_1, u_2, \dots)$ to represent partitions. The symbol $u_k$ represents $k$, and  $j u_k $ will represent $k+k+\cdots+ k$ ($j$ times). For example, $3u_1+u_2 $ represents the partition
$2+1+1+1$ of $5$. 
If we insert a $1$ by adding $u_1$, we obtain a partition $4u_1+u_2$ of $6$ and if we add $u_2$, then we obtain a partition $3u_1+2u_2$ of $3+4=7$. Thus we can define a recursive approach to list all the partitions by simply adding symbols. Observe that the symbol $u_n$ will first appear when listing partitions of $n$. 
Thus we represent partitions by 
$$\lambda=\sum_k f_k u_k,$$
where $f_k\equiv f_k(\lambda)$ are non-negative integers. 

This way of representing partitions is not standard; however it is useful when we wish to list all the partitions of $n$, upto a particular value of $n$. For example, in this notation the partitions of 
$2$ are: $2u_1$ and $u_2$, and the partitions of $3$ are: 
$3u_1, u_1+u_2, u_3$. To obtain partitions of $4$, we add $u_1$ to all the partitions of $3$ to 
obtain: $4u_1, 2u_1+u_2, u_1+u_3$. We also need to add $u_2$ to partitions of $2$, which gives an additional partition $2u_2$. (The partition obtained by adding $u_2$ to $2u_1$ is already listed.)
Finally, we add the partition represented by $u_4$. This gives us the five partitions of $4$: 
$4u_1, 2u_1+u_2, u_1+u_3, 2u_2, u_4$. 
This approach can be extended to list colored partitions and overpartitions by considering more symbols. 

The symbols 
$\lambda \vdash n$ and 
$|\lambda|=n$ are both used to say that $\lambda$ is a partition of $n$. 
If $|\lambda| = n$, we have
\begin{equation}\label{freq1}
n=\sum_k k f_k.
\end{equation}
The quantity $f_k(\lambda)$ denotes the frequency of $k$ in $\lambda$, that is, the number of times $k$ comes in $\lambda$. 
Let $$F_k(n) = \sum_{\lambda \vdash n } f_k(\lambda)$$
be the number of $k$'s appearing over all the partitions of $n$. 
By summing \eqref{freq1} over all partitions of $n$, we have
$$\sum_{k=1}^n kF_k(n) = n p(n).$$
This shows that $P(n)=p(n)$ satisfies \eqref{part-b}. 

To obtain \eqref{part-a}, observe that for $k=1, 2, \dots, n$,
\begin{equation}\label{freq-recursion}
F_k(n)=p(n-k)+F_k(n-k),
\end{equation}
because adding $u_k$ to each partition of $n-k$ yields a partition of $n$, and vice-versa, on deletion of $u_k$ from any partition containing $k$ as a part, we obtain a partition of $n-k$.  
This gives, on iteration,
$$F_k(n)=p(n-k)+p(n-2k)+p(n-3k)+\cdots,$$
 for $k=1, 2, \dots $. (We take $p(m)=0$ for $m<0$.) 
The matrix equation \eqref{part-a} follows. 

Note that \eqref{part-a} and \eqref{part-b}, together with the initial condition $P(0)=1$ determine the sequence $P(n)$ (as well as $F_k(n)$ for $k=1, 2, \dots, n $).
Since $p(n)$ satisfies the equations and the initial condition, we must have $P(n)=p(n)$. This proves our assertion earlier in the paper.

We call such a matrix the partition-frequency enumeration (PFE) matrix, since it contains all the information required to enumerate both partitions and the associated frequency function. 
 
As an immediate consequence of 
\eqref{part-a} and \eqref{part-b}, we obtain a recurrence relation that appears in Ramanujan's work
(see \cite[p.~108]{Berndt1994}), but has been credited to
Ford~\cite{ford1931}: 
\begin{equation}\label{ramanujan-rec}
\sum_{d=1}^{n} \sigma_1(d) p(n-d)=np(n).
\end{equation}
This follows by multiplying both sides of \eqref{part-a} by the diagonal matrix 
$$\text{diag}(1,2,3,\dots, n),$$ taking column sums (to get each summand on the left hand side of \eqref{ramanujan-rec}), and then, summing  the column sums using \eqref{part-b}, to obtain the right-hand side. 

Ramanujan's recurrence highlights the relevance of the matrix to the connection between arithmetical functions and partitions. There is nothing special about the arithmetic function $\sigma_1(n)$, or indeed, the partition function. This kind of analysis can be done for many different types of functions. 
In this paper, we develop a calculus for writing down matrices corresponding to many such partition-theoretic objects, and illustrate their use in finding number-theoretic results. Further, we extend the ideas to many other functions, as the following sample of results attests.

By essentially the same technique as \eqref{ramanujan-rec}, we obtain the recurrence
\begin{equation}\label{zeta-rec}
\zeta(2n)=  \frac{(-1)^{n+1} n\pi^{2n}}{(2n+1)!}  +\sum_ {k=1}^{n-1}
\frac{(-1)^{k+1} \pi^{2k}}{(2k+1)!} \zeta(2n-2k),
\end{equation}
where $n$ is a positive integer, and $\zeta(2n)$ 
is given by the series
$$\zeta(2n) = \sum_{k=1}^\infty \frac{1}{k^{2n}}.$$
This was discovered by Song~\cite{Song1988} and proved using Fourier series. 
As the reader will see below, our derivation does not involve anything more than what Euler himself used to prove \eqref{zeta-rec}
for $n=1$, that is, his evaluation
$$\zeta(2)=\sum_{k=1}^\infty \frac{1}{k^2} = \frac{\pi^2}{6}.$$

As another example, consider Ramanujan's recurrence for his $\tau$ function. 
Recall the following notation for $q$-shifted factorial (also called the $q$-Pochhammer symbol). 
For $0<|q|<1$,
we use $$\pqrfac{a}{\infty}{q}:= \prod_{k=0}^\infty (1-aq^k),$$
and the short-hand notation: 
$$\pqrfac{a_1, a_2, \dots, a_n}{\infty}{q}:= \pqrfac{a_1}{\infty}{q}\pqrfac{a_2}{\infty}{q}
\cdots \pqrfac{a_n}{\infty}{q}.$$
Ramanujan's $\tau$ function is
defined by the relation
$$\left( q; q\right)^{24}_\infty = \sum_{n=0}^\infty \tau(n+1) q^n.$$
Ramanujan's recurrence for the $\tau$ function~\cite[p.~152]{SR1916} 
\begin{equation}\label{tau-ramanujan}
(n-1)\tau(n)=\sum_{j=1}^\infty (-1)^{j+1}(2j+1)\left(n-1-9j(j+1)/2\right) \tau\big(n-j(j+1)/2\big)
\end{equation}
 looks very much like Euler's recurrence \eqref{euler-sigma}; our proofs of the two results are quite similar, too.

This paper is organized as follows. In \S\ref{sec:PFE} we give the partition-frequency enumeration (PFE) matrix for several ``partition-type objects''.  In \S\ref{sec:arithmetical-examples}, we show that the form of the matrix obtained is responsible for many relations between arithmetical functions and partition-theoretic objects.  In \S\ref{sec:calculus} we extend our calculus to find the PFE matrix for arbitrary power series and products. This extends the applicability of the techniques of this paper, and we find many results analogous to those mentioned above. 

We consider three different applications to illustrate the importance of such results. First, in \S\ref{sec:applications1}, we consider a problem of Heninger, Rains and Sloane~\cite{HRS2006} that arose in the context of taking $n$th roots of theta functions, and find related results. Next, in 
\S\ref{sec:application2}, we present some applications of a formula which Gould~\cite{Gould1974} says  ``is not as widely knows as it should be, and has been rediscovered repeatedly''.
As a  third application (in \S\ref{sec:applications3}), we embed
Ramanujan's famous congruences 
\begin{equation*}
p(5m+4)  \equiv 0 \text{ (mod $5$)} \text{ and }
\tau(5m+5) \equiv 0\text{ $($mod $5)$} ,
\end{equation*}
into an infinite family of congruences: 
$$P_r(5m+4) \equiv 0 \; (\modulo 5), \text{ if } r \equiv 1\; (\modulo 5),$$
where $P_r(n)$ are the coefficients of the power series expansion of $1/(q;q)_\infty^r$.
Here $r$ is a rational number, and a few special cases of our results overlap with congruences found by Chan and 
Wang~\cite{CW2019}. Several families of congruences of this nature appear in \S\ref{sec:applications3}.

\section{Enumeration of Partitions}\label{sec:PFE}
In this section, we find the partition-frequency enumeration (PFE) matrix for all
 ``partition-type" sequences with generating functions of the form
\begin{equation}\label{gf-gen-partition}
Q(z,q)=\prod_{k=1}^\infty\frac{1}{\left(1-zq^k\right)^{b_k}}
=\sum_{n=0}^\infty P(z, n)q^n .
\end{equation}
Here $(b_k)$ is any sequence of complex numbers 
and $|q|<1$, $|z|<1$. 
For convenience, we suppress the dependency on $z$, by writing $Q(q)$ for $Q(z,q)$, $P(n)=P(z,n)$ and so on.  
We refer the reader to \cite{Andrews1976, AAR1999, AE2004} for background  information on the theory of partitions.


We first write the system of equations in more generality.
\begin{Lemma}[An enumeration lemma] \label{lemma:enumeration} Let $\A$ be an infinite matrix, and $A_n$ be the sub-matrix formed by taking the first $n$ columns of $\A$. Suppose that $A_n$ has a finite number (say, $m_n$) of non-zero rows.  Let $U_k$ and $V_k$ be some sequences of complex numbers, with $V_k\neq 0$ for all $k$.  
Let $P(n)$, for $n=0, 1, 2, \dots$, and $F_1(n), F_2(n), \dots, F_{m_n}(n)$, be related as follows: 
\begin{subequations}
\begin{equation} \label{enumeration1-a}
A_n
\begin{pmatrix}
P(n-1)\\
P(n-2)\\
\vdots\\
P(0)
\end{pmatrix}
 = 
 \begin{pmatrix}
F_1(n)\\
F_2(n)\\
\vdots\\
F_{m_n}(n)
\end{pmatrix},
\end{equation}
and
\begin{equation} \label{enumeration1-b}
\sum_{k=1}^{m_n} U_kF_k(n) = V_nP(n).
\end{equation}
\end{subequations}
Then, given the initial condition $P(0)$, 
the above equations determine $P(n)$, and $F_k(n)$, for $k=1, 2, \dots, n$. 
\end{Lemma}
\begin{proof} The proof is immediate. 
If we know $P(0)$, $P(1)$, $\dots$, $P(n-1)$, we obtain 
$F_1(n)$, $F_2(n)$, $\dots$, $F_{m_n}(n)$ from \eqref{enumeration1-a}; then, we obtain $P(n)$ from \eqref{enumeration1-b}.
\end{proof}

\begin{Example}[Enumeration of partitions]\label{ex:partitions} We calculate the first few values of $p(n)$ by this approach, using \eqref{part-a} and \eqref{part-b}. First, for $n=1$, we see that
$$(1)\left(p(0)\right)=(1)=(F_1(1)),$$ so  $P(1) = 1$. 
Next, for $n=2$, \eqref{part-a} gives
$$
\begin{pmatrix}
1 & 1 \\
0 & 1
\end{pmatrix}
\begin{pmatrix}
p(1)\\
p(0)
\end{pmatrix}
=
\begin{pmatrix}
1 & 1 \\
0 & 1
\end{pmatrix}
\begin{pmatrix}
1\\
1
\end{pmatrix}
=
\begin{pmatrix}
2\\
1
\end{pmatrix}
=
\begin{pmatrix}
F_1(2)\\
F_2(2)\\
\end{pmatrix}.
$$
This gives $F_1(2)=2$ and $F_2(2)=1$. Thus 
$2P(2) = 1\cdot 2+2\cdot 1$, or  $P(2)=2$.  
For $n=3$, we have
$$
\begin{pmatrix}
1 & 1 &1 \\
0 & 1 & 0\\
0 & 0 &1 \\
\end{pmatrix}
\begin{pmatrix}
2\\
1\\
1
\end{pmatrix}
=
\begin{pmatrix}
4\\
1\\
1
\end{pmatrix}
=
\begin{pmatrix}
F_1(3)\\
F_2(3)\\
F_3(3)
\end{pmatrix}.
$$
From \eqref{part-b}, we have $3p(3)=1\cdot 4+2\cdot 1+3\cdot 1=9$, so $p(3)=3$.
\end{Example}
We wish to emphasize, that given the matrix $\A$, and the initial condition $p(0)=1$, we can generate $p(n)$ for $n>0$ and $F_k(n)$ satisfying \eqref{part-b}. As will become apparent, the enumeration of many partition-theoretic functions can achieved similarly, using a matrix that is not very far from \eqref{partition-matrix}.

In most of our work, we will consider the weights of Lemma~\ref{lemma:enumeration} given by $U_n=V_n=n$ and $A_n$ 
as $n\times n$ sub-matrices of an infinite matrix $\A$.  The matrix equations are of the form
\begin{subequations}
\begin{align} \label{gen-a}
A_n
\begin{pmatrix}
P(n-1)\\
P(n-2)\\
\vdots\\
P(0)
\end{pmatrix}
 &= 
 \begin{pmatrix}
F_1(n)\\
F_2(n)\\
\vdots\\
F_n(n)
\end{pmatrix},
\intertext{or, in weighted form,} 
A_n^{\prime}
\begin{pmatrix}
P(n-1)\\
P(n-2)\\
\vdots\\
P(0)
\end{pmatrix}
 &= 
 \begin{pmatrix}
F_1(n)\\
2F_2(n)\\
\vdots\\
n F_n(n).
\end{pmatrix}.
\label{gen-weighted}
\end{align}
\end{subequations}
The entries of the matrix $A_n$ are denoted $a_i(j)$ for $i, j = 1, 2, \dots, n$. The corresponding entries of 
$A_n^\prime$ are then $ia_i(j)$.   

We require notations for the generating functions of the rows of the matrix
$\A$ and the generalized frequency function. These are, respectively,
$$R_k(q) :=\sum_{j} a_k(j)q^j,
\; \text{ and } 
N_k(q) :=\sum_{n\ge 0} F_k(n)q^n.
$$

Given a PFE matrix $\A$, we refer to the corresponding sequences as the generalized partition (respectively, frequency) functions, even if such a combinatorial interpretation does not exist. 

Recall the notation for the rising factorial:
$$\rfac{a}{0}=1,\;\; \rfac{a}{r} = a(a+1)\cdots (a+r-1) \text{ for } r > 0.$$
We also require the binomial theorem \cite[Eq.\ (2.1.6)]{AAR1999}: for $|z|<1$
$$\left(1-z\right)^{-a} =\sum_{r=0}^\infty \frac{\rfac{a}{r}}{r!} z^r.$$
%
In view of the binomial theorem, we see that the generalized partition function is given by $P(0)=1$, and 
\begin{equation}\label{gen-partition-function}
P(n) = \sum_{\lambda=\sum r_i{u_i}\atop |\lambda|=n} \prod_{i}  \frac{\rfac{b_i}{r_i}}{r_i!} z^{r_i}, 
\end{equation}
where the sum is over all partitions of $n$. Note that $\sum_i r_i$ is the number of parts of the partition $\lambda$, so the power of $z$ keeps track of the number of parts. 
We define the generalized frequency function $F_k(n)\equiv F_k(z,n)$ as follows:
\begin{equation}\label{gen-freq-function}
F_k(n):= \sum_{\lambda=\sum {r_iu_i}\atop |\lambda|=n , r_k>0}  r_k \prod_{i}  \frac{\rfac{b_i}{r_i}}{r_i!} z^{r_i}. 
\end{equation}
The sum in the definition of $F_k(n)$ is over all partitions $\lambda$ of $n$ which contain $k$ as a part. We now compute the associated PFE matrix. The following theorem gives us all the elements required by Lemma~\ref{lemma:enumeration}.

\begin{Theorem}\label{th:gen-freq-matrix} Let $|z|<1$, $|q|<1$, and $(b_k)_{k\ge 1}$ be a sequence of complex numbers. Let $Q(q)$ be the generating function \eqref{gf-gen-partition}, and let $P(n)$ and $F_k(n)$ the associated functions given by \eqref{gen-partition-function} and \eqref{gen-freq-function}.
  Then:
 \begin{enumerate}[(a)]
\item  $P(0)=1$.
 \item
The matrix equation \eqref{gen-a}  holds, with entries of $A_n$ given by
\begin{subequations}
\begin{equation}
a_i(j) = 
\begin{cases}
b_iz^r & \text{if } j=ri ,\\
0 & \text{otherwise}.\\
\end{cases} 
\label{partitions-aij}
\end{equation}
Equivalently, the matrix equation \eqref{gen-weighted}  holds, with entries of $A_n^\prime$ given by
\begin{equation}
a^\prime_i(j) = 
\begin{cases}
ib_iz^r & \text{if } j=ri ,\\
0 & \text{otherwise}.\\
\end{cases} 
\label{partitions-aij-weighted}
\end{equation}
\end{subequations}
\item For $n=1, 2, 3, \dots$, we have
\begin{equation}\label{gen-column-sums}
\sum_{k=1}^n kF_k(n) = nP(n) .
\end{equation}
\end{enumerate}
\end{Theorem}
\begin{Remark} This work arose in the context of \cite{BB2020}; we wanted a convenient approach to  enumerate objects considered in that paper. 
\end{Remark}
\begin{proof} 
Clearly, from \eqref{gf-gen-partition}, we have $P(0)=1$. 

For part (b), we first prove a generalization of \eqref{freq-recursion}:
for $k=1, 2, \dots$,
\begin{equation}\label{recurrence-gen}
F_k(n)= zF_k(n-k)+b_k z P(n-k).
\end{equation}
 To show this, 
we use the elementary identity
$$r_k \frac{\rfac{b_k}{r_k}}{r_k!} = 
(r_k-1)\frac{\rfac{b_k}{r_k-1}}{(r_k-1)!}+b_k \frac{\rfac{b_k}{r_k-1}}{(r_k-1)!}
$$
for $r_k>0$. This gives
\begin{align*}
F_k(n) & = 
\sum_{\lambda=\sum {r_iu_i}\atop |\lambda|=n , r_k>0} 
  r_k  \frac{\rfac{b_k}{r_k}}{r_k!} z^{r_k} \prod_{i\neq k}  \frac{\rfac{b_i}{r_i}}{r_i!} z^{r_i} 
\\
& = 
z \sum_{\lambda=\sum {r_iu_i}\atop |\lambda|=n , r_k>0} (r_k-1)
 \frac{\rfac{b_k}{r_k-1}}{(r_k-1)!} z^{r_k-1} \prod_{i\neq k}  \frac{\rfac{b_i}{r_i}}{r_i!} z^{r_i} \\
&\hspace{1 in}+
 b_k z \sum_{\lambda=\sum {r_iu_i}\atop |\lambda|=n , r_k>0} 
 \frac{\rfac{b_k}{r_k-1}}{(r_k-1)!} z^{r_k-1} \prod_{i\neq k}  \frac{\rfac{b_i}{r_i}}{r_i!} z^{r_i} 
\\
& =
zF_k(n-k) + b_kz P(n-k).
\end{align*}
This proves \eqref{recurrence-gen}. 
Next, multiply both sides of \eqref{recurrence-gen} by $q^k$ and sum over $k$ to obtain
$$N_k(q)=zq^k N_k(q)+b_kzq^kQ(q),$$
that is,
\begin{equation}\label{freq-gen-gf}
N_k(q)=\frac{b_kzq^k}{1-zq^k} Q(q).
\end{equation}
This immediately gives an expression for the generating function of the rows of the associated PFE matrix $\A$:
\begin{equation}
R_k(q) = \frac{b_kzq^k}{1-zq^k},\label{partitions-rowgf}
 \end{equation}
Thus, for  $n\geq 1$, we have the equations \eqref{gen-a},  where the entries \eqref{partitions-aij} for the matrix are obtained by expanding the denominator of \eqref{partitions-rowgf} as a geometric series, and comparing coefficients. To obtain the entries \eqref{partitions-aij-weighted} of the equivalent form \eqref{gen-weighted}, we multiply both sides of \eqref{gen-a} by the vector $(1, 2, \dots, n)$.

Finally, \eqref{gen-column-sums} follows by multiplying \eqref{gen-freq-function} by $k$, and summing over $k$. The sum can be interchanged for each partition $\lambda=\sum_i r_i u_i$ of $n$ and we use 
$$\sum_i i r_i = n$$ to obtain the result.
\end{proof}

\begin{Examples} A few examples of partition functions are given below. They can all be enumerated using the approach of Example~\ref{ex:partitions}.
\begin{enumerate}
\item {\bf Partitions with distinct parts}: 
Partitions with distinct parts are generated by $\pqrfac{-q}{\infty}{q}$; that is, when $z=-1$ 
and $b_k=-1$ for all $k$ in \eqref{gf-gen-partition}. Equation~\eqref{recurrence-gen}
reduces to 
$$F_k(n)=P(n-k)-F_k(n-k),$$
where 
 $P(n)=p(n\; |\text{ distinct parts})$ in this context. Here we have used the self-explanatory
notation  from  \cite{AE2004} for partitions with distinct parts. 
The first row of the PFE matrix is
$$[1, -1, 1, -1, \dots ].$$
The remaining rows can be obtained from \eqref{partition-matrix} by changing signs of every-other  
non-zero entry too. 
\item {\bf Partitions with odd parts}:
The partitions with only odd parts are generated by $z=1$, $b_{2k}=0$ and $b_{2k-1}=1$ for $k>0$ in \eqref{gf-gen-partition}. The matrix obtained is the same as one obtained from \eqref{partition-matrix} by changing all entries in even-numbered rows to $0$. The corresponding frequencies $F_{2k}(n)$ all equal $0$ for all $n$.
\item {\bf Parts in $S$}: If $S$ is a subset of the natural numbers, and $P(n)$ the partitions with parts in $S$, the PFE matrix has entries
\begin{equation*}
a_i(j) = 
\begin{cases}
1 & \text{if } i\in S \text{ and } j=ri ,\\
0 & \text{otherwise}.\\
\end{cases} 
\end{equation*}
\item {\bf Colored partitions}:
Colored partitions, with part $k$ coming in $b_k$ colors, are given by the generating function \eqref{gf-gen-partition} with $z=1$ and
 $b_k\in \mathbb{N}$. 
 They are called ``prefabs" by Wilf; see \cite[\S 3.14]{wilf2006} for many examples and another intuitive interpretation of the corresponding partitions. In this case, \eqref{recurrence-gen} has a combinatorial meaning; see \cite{BB2021b}.
%
\item {\bf Plane partitions}: Plane partitions can be considered as colored partitions where the part $k$ comes in $k$ colors (see \cite[Th.\ 15, p.~101]{AE2004}), that is, when $b_k=k$ for all $k$ in \eqref{gf-gen-partition}.  
\end{enumerate}
\end{Examples}

The next result gives the PFE matrix for a product of two generating functions 
of the form \eqref{gf-gen-partition}. 
\begin{Theorem}\label{th:product-of-products}
Consider the generating function
$$\prod_{k=1}^\infty\frac{1}{\left(1-zq^k\right)^{b_k}}
\prod_{k=1}^\infty\frac{1}{\left(1-tq^k\right)^{c_k}}
=\sum_{k=1}^\infty P(n)q^n.$$
We can find the PFE matrix in the following two forms.
\begin{enumerate}[(a.)]
\item $A^{(1)}_n$ is an $n\times n $ matrix with entries given by
\begin{equation}\label{product-form1}
a_i(j) = 
\begin{cases}
b_iz^r+c_it^r & \text{if } j=ri ,\\
0 & \text{otherwise}.\\
\end{cases} 
\end{equation}
\item $A^{(2)}_n$ is a $2n\times n$ matrix with entries given by
\begin{equation}\label{product-form2}
a_i(j) = 
\begin{cases}
b_iz^r & \text{if } i=2s-1, j=rs \text{ for } s=1, 2, \dots, n,\\
c_it^r & \text{if } i=2s, j=rs \text{ for } s=1, 2, \dots, n,\\
0 & \text{otherwise}.\\
\end{cases} 
\end{equation}
\end{enumerate}
\end{Theorem}
\begin{proof}
In this case the generalized partition function is given by a sum of the form
$$P(n) = \sum_{(\lambda_1,\lambda_2)} \prod_{i}  \frac{\rfac{b_i}{r_i}}{r_i!} z^{r_i} 
\prod_{j}\frac{\rfac{c_j}{s_j}}{s_j!} t^{s_j}, $$
where the sum is over pairs of partitions $(\lambda_1,\lambda_2)$ with $|\lambda_1|+|\lambda_2|=n$; here $\lambda_1=\sum_i r_iu_i$
and $\lambda_2 =\sum_i s_i v_i$. The parts of $\lambda_1$ are accounted by powers of $z$ and parts of $\lambda_2$ by powers of $t$. The corresponding generalized frequency functions can be defined as in \eqref{gen-freq-function}. They satisfy
\begin{align*}
F_k^z(n) &= zF_k^z(n-k)+b_k z P(n-k),\\
\intertext{and}
F_k^t(n) &= tF_k^t(n-k)+c_k t P(n-k).
\end{align*}
This implies the matrix equation implied by \eqref{product-form2}.
In addition, adding the two equations we obtain
\begin{align*}
F_k^z(n)+F_k^t(n)&=zF_k^z(n)+tF_k^t(n)+(b_k z +c_kt)P(n-k)\\
&=(b_k z +c_kt)P(n-k)+(b_k z^2 +c_kt^2)P(n-2k)+\cdots.
\end{align*}
This implies the entries given in \eqref{product-form1}.
\end{proof}
\begin{Remark}\
Note that the implied matrix equations have two different columns on the right hand side of 
\eqref{gen-a} or \eqref{gen-weighted}. The choice between the two forms \eqref{product-form1} and \eqref{product-form2} is dictated by whether we are interested in computing the frequency functions $F_k^z$ and $F_k^t(n)$ individually. In either case, the result analogous to \eqref{gen-column-sums} is the same. We have
\begin{equation}\label{sumfk-prod}
\sum_{k=1}^n k\left(F_k^z(n)+F_k^t(n)\right) = nP(n) .
\end{equation}
\end{Remark}

As a prototypical example, we consider the overpartitions introduced by Corteel and Lovejoy \cite{CL2004}.  
\begin{Example}[Overpartitions] An overpartition of $n$ is a non-increasing sequence of natural numbers whose sum is $n$, where the first occurrence of a  number can be overlined. We denote the number of overpartitions of $n$ by
$\overline{p}(n)$. 
The 
generating function of overpartitions is the product of the respective generating functions: 
$$\overline{Q}(q)=
\sum_{n\geq 0} \overline{p}(n)q^n  
=
\frac{\pqrfac{-q}{\infty}q}{\pqrfac{q}{\infty}q}.$$
We obtain the PFE matrix from \eqref{product-form2} by taking $z=-1$, $b_j=-1$, $t=1$ and $c_j=1$. The matrix consists of rows of the PFE matrix corresponding to distinct partitions alternating with the PFE matrix for ordinary partitions. 
\end{Example}

We have seen that the enumeration of many types of partitions can be accomplished using PFE matrices very similar to \eqref{partition-matrix}. As a bonus, we obtain information about the frequency function associated with partitions. Next, we explore the connection to arithmetic functions; this is an immediate consequence of the form of the PFE matrix. 

\section{Connection with arithmetic functions}\label{sec:arithmetical-examples}
The objective of this section is to illustrate the connection between the arithmetic functions of 
number theory and partition-theoretic functions. The connection is due to the following proposition,
which applies whenever we have a PFE matrix with entries a mild variation of the entries of the matrix \eqref{partition-matrix}. 

\begin{Proposition}\label{prop:divisor-sums}
Let $P(n)$ and $F_k(n)$ satisfy $P(0)=1$, 
the matrix equations \eqref{gen-a} with entries $a_i(j)$ given by \eqref{partitions-aij} (or equivalently,
\eqref{gen-weighted} with entries $ia_i(j)$), and \eqref{gen-column-sums}.
Let $f:\mathbb{N}\to \mathbb{C}$ and $g(n)$ be defined as
\begin{equation}\label{gn}
g(n):=\sum_{d\mid n} b_d f(d) z^{n/d}.
\end{equation}
Then we have the following identities:
\begin{subequations}
\begin{align}
\sum_{k=1}^n g(k) P(n-k) &= \sum_{k=1}^n f(k)F_k(n); \label{divisor-sum}\\
Q(q)\sum_{k=1}^\infty {g(k)} q^k &=\sum_{n=0}^\infty \Big(\sum_{k=1}^n f(k)F_k(n)\Big) q^n.\label{divisor-sum-gf1} 
\end{align}
\end{subequations}
\end{Proposition}

%
%
\begin{Remarks}\
\begin{enumerate}
\item When $f(k)=k$, then in view of \eqref{gen-column-sums}, equations \eqref{divisor-sum} and \eqref{divisor-sum-gf1} reduce to
\begin{subequations}
\begin{align}
\sum_{k=1}^n g(k) P(n-k) &= nP(n); \label{divisor-sum-b}\\
Q(q)\sum_{k=1}^\infty {g(k)} q^k &=\sum_{n=0}^\infty nP(n)q^n.
\label{divisor-sum-gf1-b} 
\end{align}
\end{subequations}
\item  When $z=1$, and $g(n)$ and $b_kf(k)$ are both arithmetic functions, then
we obtain results connecting number-theoretic functions to partition-theoretic functions. 
\item The sequence $g(n)$ is the sequence of the weighted sums of the $n$th column of the matrix given
in \eqref{partitions-aij}; more precisely, $g(n)$ is the $n$th column sum of the matrix 
$$(f(1), f(2), \dots, f(n)) A_n.$$ 
\item Many results that fit the form of Proposition~\ref{prop:divisor-sums} were previously obtained by taking logarithmic derivatives of the relevant generating functions. 
Thus the enumeration result in Theorem~\ref{th:gen-freq-matrix} can be considered an
alternate approach to this standard technique.
\end{enumerate}
\end{Remarks}
\begin{proof}
To obtain \eqref{divisor-sum}, multiply the $k$th row by $f(k)$ on both sides of the equation \eqref{gen-a}. The $k$th row on either side is
$$\sum_{j>0} f(k) a_k(j) P(n-j) = f(k)F_k(n),$$
where $a_i(j)$ is given by \eqref{partitions-aij}. Now summing over $k$,
and noting that $a_k(j)= b_k z^{j/k}$ when $k\mid j$ and $0$ otherwise, we obtain the result.

To obtain \eqref{divisor-sum-gf1}, we multiply both sides of \eqref{divisor-sum} by $q^n$ and sum over $n$. 
\end{proof}

Let $P(n)$ and $Q(q)=Q(z,q)$ be given by \eqref{gf-gen-partition}. By Theorem~\ref{th:gen-freq-matrix}, the conditions of the proposition are satisfied, and we have the equations  \eqref{divisor-sum} and\eqref{divisor-sum-gf1}. 
In the examples below, we illustrate the application of these equations to obtain several classical and some new results.  In addition, we use \eqref{freq-gen-gf} in the form:
\begin{equation}
\frac{b_k z q^k}{1-zq^k} Q(q) =  \sum_{n=0}^\infty F_k(n) q^n. \label{row-equation}
\end{equation}

\begin{Example}[Variations of Ramanujan's recursion]
Ramanujan's recurrence  \eqref{ramanujan-rec} follows from \eqref{divisor-sum}. Let $z=1$, $b_k=1$ and $f(k)=k$, for all $k$, so that $g(n)=\sigma(n)$. 
There are several ways one can find analogous results. Let $\sigma_k(n)$ now denote the sum of $k$th powers of the divisors of $n$ (with $\sigma_1(n)=\sigma(n)$).  In  the following, we take $z=1$. 

First take $b_k=k$ and $f(k)=k$. Then $g(n)=\sigma_2(n)$, the sum of squares of the divisors of $n$. We obtain a well-known recursion for plane partitions (denoted $PL(n)$):
$$nPL(n)=\sum_{d=1}^{n} \sigma_2(d) PL(n-d).$$
See, for example, \cite[Ex.~7, pg.~28]{Mac1995}. Alternatively, let $b_k=r$, so we obtain 
colored partitions where each part $k$ comes in $r$ colors. In this case, we obtain
$$np_r(n)=r\sum_{d=1}^{n} \sigma_1(d) p_r(n-d),$$
where we have denoted the number of associated partitions by $p_r(n)$. This gives a solution to
Apostol~\cite[Ex.~9, Ch.~14]{Apostol1976}. (However, our solution is valid even when $r$ is not a positive integer.)

Another possibility is to take $f(k)=k^m$ (keeping $b_k=1$) for all $k$.  We obtain
an expression for the $m$th moment of the frequency function corresponding to the ordinary partitions:
\begin{equation}\label{def:moments}
M_m(n):=\sum_{k=1}^n k^m F_k(n)=\sum_{d=1}^{n} \sigma_m(d) p(n-d).
\end{equation}

Finally, another variation is obtained by taking different arithmetic functions. For example,
consider the M\"{o}bius function $\mu$ defined by
\begin{equation}\label{mu-def}
\mu(n) := 
\begin{cases}
1 & \text{if } n =1,\\
(-1)^k & \text{if } n=p_1p_2p_3\dots p_k, \\
0 & \text{otherwise.} \\
\end{cases} 
\end{equation}
It is well known \cite[p.~25]{Apostol1976} that 
\begin{equation}\label{sum-mu}
\sum_{d\mid n} \mu (d) = 
\begin{cases} 1 & \text{if } n=1,\\
0 & \text{otherwise}.
\end{cases}
\end{equation}
Taking the case $b_k=1$, $f(k)=\mu(k)$ gives
\begin{equation}
\sum_{k=1}^n \mu(k) F_k(n) = p(n-1),
\end{equation}
which indicates the close relationship of $F_k(n)$ with the partition function. 
Analogous results can be obtained for $b_k=r$ (a constant), and, when $f(k)$ is any other interesting arithmetic function (such as Euler's totient function or Liouville's function \cite[p.~37]{Apostol1976}), where there are identities of the form
$$g(n)=\sum_{d\mid n} f(d).$$
\end{Example}

\begin{Example}(\cite[p.~327]{Apostol1976})\label{exp1} We now obtain a nice result regarding the connection of $\mu(n)$ with the exponential function. We want to use \eqref{sum-mu} to obtain a nice 
expression for $g(n)$. Take $z=1$, $b_k=\mu(k)/k$ and $f(k)=k$ in \eqref{gn} to obtain
$$g(n)=
\begin{cases} 1 & \text{if } n=1,\\
0 & \text{otherwise}.
\end{cases}$$
Then \eqref{divisor-sum-b} reduces to
$$nP(n) =\sum_{k=1}^n g(k) P(n-k) = P(n-1).$$
This gives, on iteration, $P(n)=1/n!$, and 
$$Q(q)=\sum_{n=0}^\infty \frac{q^n}{n!} =e^q;$$ or, 
$$\prod_{k=1}^\infty \big({1-q^k}\big)^{-\frac{\mu(k)}{k}} = e^q.
$$
This is valid for $|q|<1$, and as formal power series. 
\end{Example}

The next example uses a technique of Euler. 
We use Euler's pentagonal number theorem:  
\begin{equation}\label{PNT}
\pqrfac{q}{\infty}{q} = \sum_{k=-\infty}^\infty (-1)^k q^{\frac{k(3k-1)}{2}} 
= \sum_{n=0}^\infty \euler(n)q^n,
\end{equation}
where 
$$\euler(n)=\begin{cases}
(-1)^k & \text{if }  n=\frac{k(3k\pm 1)}{2},\\
0 & \text{otherwise.} 
\end{cases}
$$
Recall that $\euler(n)$ is the number of partitions of $n$ consisting of an even number of distinct parts minus the number of partitions of $n$ with an odd number of distinct parts. 

\begin{Example}(Moments of the frequency function)\label{ex:moments} We extend some formulas of Andrews and Mirca \cite{AM2020}.  Let $M_m(n)$, defined in \eqref{def:moments}, denote the $m$th moment of the frequency function corresponding to partitions.
Take $b_k=1$ and $f(k)=k^m$ for all $k$.
Then \eqref{divisor-sum-gf1} can be written as
\begin{equation}
\sum_{n=0}^\infty M_m(n)q^n=\sum_{n=0}^\infty \sum_{k=1}^n k^m F_k(n)q^n =
\frac{1}{\pqrfac{q}{\infty}{q}}\sum_{n=0}^\infty\sigma_m(n)q^n .
\end{equation}
Multiplying both sides by 
$\pqrfac{q}{\infty}{q}$, using \eqref{PNT} and comparing coefficients of $q^n$ on both sides,
we obtain
\begin{equation}\label{moments1}
\sigma_m(n)=\sum_{j=-\infty}^\infty (-1)^j M_m \left(n- j(3j-1)/2\right).
\end{equation}
For $m=0$ and $m=1$, we obtain results of Andrews and Merca \cite{AM2020}. 

A related result is obtained by considering \eqref{row-equation} with $z=1$, $b_k=1$ and $f(k)=1$  (so $m=0$ in the above).
We find that
$$\sum_{n=1}^\infty q^{kn} = \pqrfac{q}{\infty}{q}\sum_{n=0}^\infty F_k(n) q^n.$$
On comparing the coefficients of $q^n$ on both sides we find that 
\begin{equation}
\sum_{j=-\infty}^\infty (-1)^j F_k \left(n- j(3j-1)/2\right)  =
\begin{cases}
1 & \text{if } k\mid n, \\
0  & \text{otherwise.}
\end{cases}
\end{equation}
This is a refinement of the last formula of \cite{AM2020}. Evidently  $\sum_k F_k(n)$ is the total number of parts occurring in all the partitions of $n$. If we sum over $k$, we obtain their formula. 
\end{Example}

Next we obtain recurrences for the divisor function $\sigma_1(n)$, including Euler's recurrence \eqref{euler-sigma}. What is different from the previous example is that we consider different partition-theoretic functions. 
\begin{Example}[Recurrences for $\sigma_1(n)$]
To obtain \eqref{euler-sigma}, we take $b_k=-1$ for all $k$. The role of partitions is played by $\euler(n)$
in Euler's formula \eqref{PNT}.  We take $f(k)=k$, so that $g(n)=-\sigma_1(n)$. Now 
\eqref{divisor-sum-gf1-b} and \eqref{PNT} give
$$\pqrfac{q}{\infty}{q} \sum_{n=1}^\infty (-1) \sigma_1(n)q^n  =
\bigg(\sum_{k=-\infty}^\infty (-1)^k q^{\frac{k(3k-1)}{2}} \bigg) 
\bigg( \sum_{n=0}^\infty (-1)\sigma_1(n)q^n \bigg)
=\sum_{n=0}^\infty n\euler(n)q^n.$$
On comparing coefficients of $q^n$ we obtain \eqref{euler-sigma} in the form
\begin{equation}\label{euler-sigma2}
\sum_{j=-\infty}^\infty (-1)^{j+1}\sigma_1\big(n-j(3j-1)/2\big) = n\euler(n).
\end{equation}
Note that to match the two forms, we take $\sigma_1(0)=n$ (if it occurs) in \eqref{euler-sigma}.

As a variation, we now use a result of Jacobi~\cite[p.~500]{AAR1999}:
\begin{equation}\label{jacobi1}
\left(q;q\right)^3_\infty = \sum_{k=0}^\infty (-1)^k (2k+1) q^{\frac{k(k+1)}{2}}.
\end{equation}
Take $b_k=-3$ and $f(k)=k$. Here $Q(q)$ reduces to $\left(q;q\right)^3_\infty$, and we use \eqref{jacobi1} rather than \eqref{PNT}
to expand the products. On comparing coefficients as before, we obtain a 
result of Ewell~\cite[Th.\ 3]{Ewell1977}:
\begin{equation*}
\sum_{k=0}^\infty (-1)^k (2k+1)\sigma_1\big( n-k(k+1)/2\big) 
=
\begin{cases} 
(-1)^{k+1}  \frac{k(k+1)(2k+1)}{6} & \text{if } n =\frac{k(k+1)}{2},\\
0 & \text{otherwise}.
\end{cases}
\end{equation*}

Another interesting identity is obtained by noting that
\begin{align*}
\pqrfac{q}{\infty}{q} \sum_{n=0}^\infty \sigma_1(n)q^n  &=\sum_{n=0}^\infty n\euler(n)q^n,\\
\intertext{and}
-\frac{1}{\pqrfac{q}{\infty}{q}} \sum_{n=0}^\infty \sigma_1(n)q^n  &=\sum_{n=0}^\infty np(n)q^n.
\end{align*}
On multiplying the two and comparing the coefficient of $q^n$ on both sides, we obtain
an identity for the convolution of the $\sigma$ function in terms of the partition function
\begin{multline}
\sum_{k=1}^{n-1} \sigma_1(k)\sigma_1(n-k) \\
=
\sum_{j=-\infty}^\infty (-1)^{j+1}\bigg(\frac{j(3j-1)}{2}\bigg)\bigg(n-\frac{j(3j-1)}{2}\bigg)  p\Big(n-\frac{j(3j-1)}2\Big). 
\end{multline}
For another recurrence relation for $\sigma_1(n)$, see Example~\ref{ex:JTP-prod}.
\end{Example}

We have now seen several techniques to obtain results connecting certain arithmetic functions of elementary number theory with partition-theoretic objects. Next, we extend the applicability of these techniques by considering arbitrary generating functions. 

\section{A calculus for obtaining PFE Matrices}\label{sec:calculus}

We have seen in Section~\ref{sec:PFE} that given the PFE matrix, we can determine $P(n)$ and $F_k(n)$. In addition, we took $P(n)$ to be generated by infinite products of the form \eqref{gf-gen-partition}, or products thereof, and computed the matrix.
We now consider any power series of the form 
\begin{equation}\label{power-series}
Q(q)=\sum_{n=0}^\infty P(n) q^n
\end{equation}
and find a PFE matrix associated with it. 
This will enhance the applicability of the techniques of Section~\ref{sec:arithmetical-examples}.

\begin{Theorem}\label{th:series-PFE}
Let $P(n)$ be a given sequence with $P(0)=1$. 
 Then there is a unique sequence
$(b_n)$ and a frequency function $F_k(n)$, for $k=1, 2, \dots, n$, 
satisfying \eqref{gen-column-sums} and matrix equations of the form \eqref{gen-weighted},
where the entries of the sequence of $n\times n$ matrices $A^\prime_n$ are given by
$$a_i(j) = 
\begin{cases}
i b_i & \text{if } i\mid j ,\\
0 & \text{otherwise}.\\
\end{cases} \label{wieghted-aij} 
$$
\end{Theorem}
\begin{Remarks}
\
\begin{enumerate} 
\item The $P(n)$ in this theorem may depend on $z$ too, and $b_j$ may depend on $z$. Thus the matrix obtained may be different from that obtained in Theorem~\ref{th:gen-freq-matrix}.
\item Proposition~\ref{prop:divisor-sums} (with $z=1$) applies to the case where $Q(q)$ is given as a power series. 
\item Theorem~\ref{th:product-of-products} applies to the case where the generating function is given by a product of series, too. 
\end{enumerate}
\end{Remarks}
\begin{proof}
The proof is by induction. Note that \eqref{gen-column-sums} for $n=1$ gives $F_1(1)=p(1)$. 
Now the entries of the first row of $A_n$ (which are all equal to $a_1(1)=b_1$) can be determined 
from \eqref{gen-weighted}. Next if $b_1$, $b_2$, $\dots$, $b_{n-1}$ are known,  \eqref{gen-weighted} gives the values $F_1(n)$, $F_2(n)$, $\dots$, $F_{n-1}(n)$.  Use \eqref{gen-column-sums} to determine $F_n(n)$, and then \eqref{gen-weighted} again to find $b_n$. 
\end{proof}

We illustrate the use of this theorem by considering an example from 
Andrews \cite[Problem 12-2]{Andrews1971} (see also \cite[p.~105]{Andrews1986}, \cite[p.~31]{AE2004} and \cite{GB2015}).
\begin{Example}[How to discover the (first) Rogers--Ramanujan identity] Consider $P(n)$ to be the set of partitions of $n$ with the property that the 
 difference between parts is at least $2$. The initial values for $n=0, 1, 2, \dots$ of $P(n)$ 
are:
$1$, $1$, $1$, $1$, $2$, $2$, $3$, $3$, $4$, $5$, $6$, $7$, $9$, $10$, $12$, $14$. 
From here we compute the matrix using Theorem~\ref{th:series-PFE}. 

We find that $F_1(1)=1$. This gives $b_1=1$. 
Next, for $n=2$, \eqref{part-a} gives
$$
\begin{pmatrix}
1 & 1 \\
0 & 2b_2
\end{pmatrix}
\begin{pmatrix}
1\\
1
\end{pmatrix}
=
\begin{pmatrix}
2\\
2b_1
\end{pmatrix}
=
\begin{pmatrix}
F_1(2)\\
2F_2(2)\\
\end{pmatrix}.
$$
This gives $F_1(2)=2$. Now, from \eqref{gen-column-sums} we get $2+2F_2(2)=2$,
so $2F_2(2)=2b_2=0$. Thus,  $b_2 = 0$.  
For $n=3$, we have
$$
\begin{pmatrix}
1 & 1 &1 \\
0 & 0 & 0\\
0 & 0 &3b_3 \\
\end{pmatrix}
\begin{pmatrix}
1\\
1\\
1
\end{pmatrix}
=
\begin{pmatrix}
3\\
0\\
3F_3(3)
\end{pmatrix}.
$$
This gives $3P(3)=1\cdot 3+2\cdot 0+3\cdot F_3(3)=3$, so $F_3(3)=0=b_3$. The next step gives
$$
\begin{pmatrix}
1 & 1 &1 &1 \\
0 & 0 & 0& 0\\
0 & 0 &0 & 0 \\
0& 0 & 0 & 4b_4
\end{pmatrix}
\begin{pmatrix}
1\\
1\\
1\\
1\\
\end{pmatrix}
=
\begin{pmatrix}
4\\
0\\
0\\
4b_4
\end{pmatrix}.
$$
This gives $F_1(4)=4$, $F_2(4)=F_3(4)=0$ and $4+4b_4=4 \cdot 2$, which gives $b_4=1$.
Carrying on in this fashion, we discover that the first few values of $b_k$ are given by the sequence
$$b_k=\begin{cases}
1 & k \equiv 1 \text{ or } 4 \text{ (mod $5$)},\\
0 & \text{otherwise.}
\end{cases}
$$
Assuming this pattern continues,  we see that the matrix obtained is the same as the one corresponding to 
the product $$\prod_{k=0}^\infty \frac{1}{(1-q^{5k+1})(1-q^{5k+4})},$$
which generates partitions with parts congruent to $1$ or $4$ (mod $5$).  
This shows how one can discover the (first) Rogers--Ramanujan identity. 
\end{Example}
This example shows that we can use
Theorem~\ref{th:series-PFE} to experimentally obtain an infinite product of the form
\begin{equation}\label{sum-prod}
1+\sum_{k=1}^\infty P(n)q^n=\prod_{k=1}^\infty \big(1-q^k\big)^{-b_k}
\end{equation}
corresponding to a power series. 
Andrews~\cite[Theorem~10.3]{Andrews1986} has given an algorithm (usually called Euler's algorithm) for this purpose. The relationship of 
Theorem~\ref{th:series-PFE} with Andrews' approach will become clear with the following theorem, which reverses the process we have followed so far.
Given an arithmetic sequence $g(n)$, we find an associated PFE matrix together with the associated partition and frequency functions. 

\begin{Theorem}\label{th:PFE-divisor-sums}
Let $g(n)$, $n=1, 2, 3, \dots$, be any sequence of complex numbers. Then there are sequences $P(n)$, for $n=1, 2, \dots$, with $P(0)=1$; $F_k(n)$ for $k=1, 2, \dots, n$; and a sequence $b_n$, satisfying 
\eqref{gen-column-sums} and the matrix equation 
\eqref{gen-weighted}, with entries given by
\begin{equation*}
a^\prime_i(j) = 
\begin{cases}
ib_i & \text{if } j=ri ,\\
0 & \text{otherwise},
\end{cases}
\end{equation*}
such that,
\begin{equation}\label{divisor-rec}
nP(n)=\sum_{k=1}^n g(k)P(n-k).
\end{equation}
In addition, the sequences $b_n$, $P(n)$ and $F_k(n)$ are uniquely determined. 
\end{Theorem}
\begin{proof}
\begin{subequations}
We define $b_n$ by means of the equation
\begin{equation}\label{gn-from-bn}
g(n)=\sum_{d\mid n} d b_d;
\end{equation}
so, by M\"{o}bius inversion \cite[p.~30]{Apostol1976}, we have 
\begin{equation}\label{bn-from-gn}
nb_n = \sum_{d\mid n} \mu\big(\frac{n}{d}\big) g(d),
\end{equation}
\end{subequations}
where $\mu(n)$ is defined in \eqref{mu-def}. Next take $Q(q)$ as in \eqref{gf-gen-partition}  with $z=1$
to  obtain the associated partition function $P(n)$ and the associated frequency function $F_k(n)$. Theorem~\ref{th:gen-freq-matrix}  implies they satisfy the matrix equation \eqref{gen-weighted} with entries as given, and 
\eqref{gen-column-sums}. Further, Proposition~\ref{prop:divisor-sums} (with $z=1$) implies
\eqref{divisor-rec}. Finally, note that \eqref{bn-from-gn} determines $b_n$, and by Lemma~\ref{lemma:enumeration}, the sequences $P(n)$, $F_k(n)$ are uniquely determined. 
\end{proof}

The aforementioned algorithm of Andrews~\cite[Theorem~10.3]{Andrews1986}  for finding the product form 
in  \eqref{sum-prod} 
combines  \eqref{divisor-rec} and \eqref{gn-from-bn} in the following two forms:
\begin{subequations}
\begin{align}\label{divisor-recurrence-a}
nP(n) &=  nb_n+\sum_{d\mid n \atop d<n} db_d+\sum_{k=1}^{n-1} g(k)P(n-k) \\
&= nb_n+\sum_{d\mid n \atop d<n} db_d+\sum_{k=1}^{n-1} \sum_{d\mid k} db_d P(n-k).
\label{divisor-recurrence-b}
\end{align}
\end{subequations}
One difference between Andrews' approach and that of this paper 
is that we do not restrict the use of 
Euler's algorithm to the case where $P(n)$ and $b_k$ in \eqref{sum-prod} are integers. 

Next, we give a few examples to illustrate a calculus for finding the PFE matrix for various kinds of functions. First we  show how one can find the PFE matrix for finite products. 

\begin{Example}[The complete symmetric function] We compute the PFE matrix for the complete homogeneous symmetric polynomial $h_k(x)$ in $m$ variables $x=(x_1, \dots, x_m)$. Take
$$Q_m(q)=
\prod_{k=1}^m \frac{1} { 1-x_k q} =\sum_{n=0}^\infty h_n(x) q^{n}.$$
For fixed $k$, 
the PFE matrix for $1/(1-x_kq)$ is the row matrix
$$
\left[ x_k, x_k^2, x_k^3, \dots \right].$$ 
This follows from Theorem~\ref{th:gen-freq-matrix} with $z\mapsto x_k$. 
In view of Theorem~\ref{th:product-of-products}, applied in the form \eqref{product-form2}, we obtain for any fixed $m$, and $n=1, 2, 3, \dots$ an $m\times n$ matrix $A_n^m$ with $(i,j)$th entry $x_i^j$. 

The weighted sums $g(n)$ from \eqref{gn} reduce to $p_m(x)$, the power symmetric functions, and, in view of 
\eqref{sumfk-prod}, equation \eqref{divisor-sum} reduces to \cite[eq.\ 2.11]{Mac1995}:
$$\sum_{k=1}^n p_k(x)h_{n-k}(x)=nh_n(x).$$

In this example, we can increase the number of variables in $x$ indefinitely, as is routinely done in the theory of symmetric functions. 
\end{Example}

There is another way to take a limiting process, which we illustrate by the following example.
\begin{Example}[A recurrence for $\zeta(2n)$]\label{ex:zeta} We compute the PFE matrix of
$$\frac{\sin \pi z}{\pi z} =\prod_{k=1}^\infty \Big( 1-\frac{z^2}{k^2}\Big) 
=\sum_{k=0}^\infty \frac{(-1)^k \pi^{2k} z^{2k}}{(2k+1)!}
$$
in two different ways. 

First note that for any fixed $k$, the PFE matrix for $(1-z^2/k^2)$ is the row matrix
$$
\left[ -\frac{1}{k^2}, -\frac{1}{k^4}, -\frac{1}{k^6}, \dots \right],$$
where the variable $q$ is taken to be $z^2$.  
This follows from Theorem~\ref{th:gen-freq-matrix} with $z\mapsto 1/k^2$, $q\mapsto z^2$. Now let
$$Q_m(z^2) = \prod_{k=1}^m \Big( 1-\frac{z^2}{k^2}\Big) =\sum_{n} P_n^m z^{2n}.$$
In view of Theorem~\ref{th:product-of-products}, applied in the form \eqref{product-form2}, we obtain, for any fixed $m$, and $n=1, 2, 3, \dots$, the $m\times n$ matrix $A_n^m$ with $(i,j)$th entry
$$a_i(j)=-\frac{1}{i^{2j}}.$$
This set of matrices can be used to compute $P_n^m$ using Lemma~\ref{lemma:enumeration}
(take $A_n\mapsto A_n^m, U_k=1, V_k=k, m_n=m$). 
The function defined by \eqref{gn} 
is given by 
$$g_m(n)=(-1)\sum_{k=1}^m \frac{1}{k^{2n}},$$
and from \eqref{sumfk-prod} and \eqref{divisor-sum} we obtain 
$$\sum_{k=1}^n g_m(k) P_n^m(n-k) = nP_n^m(n).$$
As $m\to\infty$, this becomes a formula due to Song \cite{Song1988}:
$$\sum_{k=1}^n (-1) \zeta(2k)a_{2(n-k)} = na_{2n},$$ 
where
$$a_{2n} = (-1)^n\frac{\pi^{2n}}{(2n+1)!}.$$
The coefficients $a_{2n}$ are obtained from the coefficients of the power series for $\sin{\pi z}/{\pi z}$. We can rewrite this formula  in the form \eqref{zeta-rec} mentioned in the introduction. 

On the other hand, taking 
$g(n)=-\zeta(2n)$,
we can obtain a product of the form
$$\prod_{k=1}^\infty \big(1-z^{2k}\big)^{-b_k},$$
where, from \eqref{bn-from-gn},
$$k b_k = - \sum_{d\mid k} \mu\big( {k}/{d}\big) \zeta(2d).$$
Thus we must have
\begin{equation}\label{alt-zeta}
\frac{\sin \pi z}{\pi z} = \prod_{k=1}^\infty \big(1-z^{2k}\big)^{\frac{1}{k}\sum_{d\mid k} \mu( {k}/{d}) \zeta(2d)}.
\end{equation}
This formula is valid as a formal power series in $z$. As an analytic formula, some further justification is required, which we do not provide here. 
\end{Example}
We can extend this idea and write the PFE matrix of any function which can be written as a Weierstrass product (see \cite[Th.\ 7, p.~194]{Ahlfors1966}). This requires us to calculate a PFE matrix for the exponentials of polynomials.  

We require a simple trick. Note that if $g(n)$ is a given function, then one PFE matrix that corresponds to it is the row matrix
$$\big[ g(1), g(2), g(3), \dots \big].$$
This is easy to verify. However, this may not be the same matrix as generated by Theorem~\ref{th:PFE-divisor-sums}.

We re-look at Example~\ref{exp1} to obtain a PFE matrix for $e^{aq}$. Note that the coefficients $P(n)$ in the power series expansion of 
$e^{aq}$ satisfy $$nP(n)=aP(n-1).$$
Comparing with \eqref{divisor-rec}, we see that $g(n)$ defined as follows works:
$$g(n)=\begin{cases}
a & \text{for } n=1,\cr
0 & \text{for } n>1.
\end{cases}
$$
So one PFE matrix generating $e^{aq}$ is 
$$\big[ a, 0, 0, \dots \big].$$
Similarly, one can generate a matrix for $e^{\frac{a}{m}q^m}$ by taking a row matrix with $a$ in the $m$th place and $0$s elsewhere. To obtain the matrix for the exponential of a polynomial in $q$, we use Theorem~\ref{th:product-of-products}. 

\begin{Example}[Gamma Function] The Gamma function has a Weierstrass product given by
$$\frac{1}{\Gamma(z)} = ze^{\gamma z} \prod_{k=1}^\infty \Big( 1+\frac{z}{k}\Big) e^{-z/n},$$
where $\gamma$ is the Euler-Mascheroni constant. 

We consider 
$$Q_m(z) = \prod_{k=1}^m \Big( 1+\frac{z}{k}\Big)e^{-z/k}.$$
We take the variable $z$ in place of $q$. 
Note that for any fixed $k$, the PFE matrix for $(1+z/k)e^{-z/k}$ is the row matrix
$$
\left[0, -\frac{1}{k^2}, \frac{1}{k^3}, -\frac{1}{k^4}, \dots \right].$$ 
Here we have added $-1/k$ in the first column to account for the factor $e^{-z/k}$. 
In view of Theorem~\ref{th:product-of-products}, applied in the form \eqref{product-form2}, we obtain for any fixed $m$, and $n=1, 2, 3, \dots$ and $m\times n$ matrix $A_n^m$ with $(i,j)$th entry
$$a_i(j)=
\begin{cases}
0 & \text{if } j=1, \cr
(-1)^{j-1}\frac{1}{i^{j}} &\text{for } j>1.
\end{cases}
$$
The corresponding function defined in \eqref{gn} is given by
$$g_m(n)=
\begin{cases}
0 & \text{if } n=1, \cr
(-1)^{n-1}\sum_{k=1}^m\frac{1}{k^{n}} &\text{for } n>1 .
\end{cases}
$$
As in Example~\ref{ex:zeta}, we can take limits as $m\to\infty$. In addition, we can obtain an alternate product formulation for 
$$\frac{1}{z\Gamma(z)e^{\gamma z}}$$
by applying Theorem~\ref{th:PFE-divisor-sums} to
$$g(n)=
\begin{cases} 0 & \text{if } n=1,\\
(-1)^{n-1}\zeta(n), & \text{for } n>1. 
\end{cases}
$$
\end{Example}

\begin{Example}[Jacobi Triple Product]\label{ex:JTP-prod} We now consider the product side of the Jacobi Triple Product identity \cite[p.~10]{Berndt2006}: 
$$1+\sum_{n=1}^\infty\big( z^{n}+z^{-n}\big) q^{n^2} = \pqrfac{q^2, -zq, -q/z}{\infty}{q^2}.$$
Here $P(n)$ is given by
$$P(n)=
\begin{cases}
z^m +z^{-m} & \text{if } n=m^2,\\
0 & \text{otherwise}.
\end{cases}
$$
We can compute the PFE matrix from the product side. 
A PFE matrix corresponding to the product $\pqrfac{q^2}{\infty}{q^2}$ has the $(i,j)$ entry given by:
$$
\begin{cases}
-1 & \text{if $i$ is even, and } j=ri \text{ for some } r,\\
0 & \text{otherwise}.
\end{cases}
$$
A PFE matrix corresponding to the product $\pqrfac{-qz,-q/z}{\infty}{q^2}$ has entries
$$
\begin{cases}
(-1)^{r-1}\big(z^r+z^{-r}\big) & \text{if $i$ is odd, and } j=ri \text{ for some } r,\\
0 & \text{otherwise}.
\end{cases}
$$
Thus the PFE matrix of the product $\pqrfac{q^2, -zq, -q/z}{\infty}{q^2}$ is given by
$$a_i(j) = 
\begin{cases}
(-1)^{r-1}\big(z^r+z^{-r}\big) & \text{if $i$ is odd, and } j=ri \text{ for some } r,\\
-1 & \text{if $i$ is even, and } j=ri \text{ for some } r,\\
0 & \text{otherwise}.
\end{cases}
$$

We now apply this for the special case $z=1$ in Jacobi's triple product identity, which corresponds to Gauss' identity \cite[Eq.\ (1.3.13)]{Berndt2006}:
\begin{equation}
\varphi(q) =\sum_{k=-\infty}^\infty q^{k^2} = 1+2\sum_{k=1}^\infty q^{k^2} =\pqrfac{q^2}{\infty}{q^2} \left(-q;q^2\right)^2_\infty .  \label{Gauss1}
\end{equation}
Here the corresponding partition function is given by 
$$P(n)
=\begin{cases}
1 & \text{for } n=1,\\
2 & \text{for } n=k^2, \text{ for some }k,\\
0 & \text{otherwise}.
\end{cases}
$$
Take $f(k)=k$ to find that $g(n)$ (in \eqref{gn}) is given by
$$g(n)=
\begin{cases}
2\sigma_1(2m-1) & \text{for } n= 2m-1, m=1, 2, \dots ,\\
-\big( 2\sigma_1^o(2m) +\sigma_1^e(2m) & \text{for } n=2m, m=1, 2, \dots;
\end{cases}
$$
where $\sigma_1^o(n)$ (respectively, $\sigma_1^e(n)$) denotes sum of odd divisors (respectively, even divisors) of $n$. Using the elementary observations
\begin{align*}
\sigma_1^o(2m) &= \sigma_1(2m) - \sigma_1^e(2m),\\
\intertext{and}
\sigma_1^e(2m) &= 2\sigma_1(m),
\end{align*}
we find that
$$g(n)=
\begin{cases}
2\sigma_1(2m-1) & \text{for } n= 2m-1, m=1, 2, \dots ,\\
- 2\sigma_1(2m) + 2\sigma_1(m) & \text{for } n=2m, m=1, 2, \dots.
\end{cases}
$$
With $Q(q)=\varphi(q)$, and $g(n)$, $P(n)$ as given here, from \eqref{divisor-sum-gf1-b}, we
obtain
$$\Big(1+2\sum_{n=1}^\infty q^{n^2} \Big) \sum_{k=1}^{\infty} g(k)q^k = 2\sum_{j=0}^\infty j^2 q^{j^2}
.$$
For $k>0$, we compare coefficients of $q^k$ to obtain
$$
\frac{1}{2}g(k) + \sum_{j=1}^\infty g(k-j^2) =
\begin{cases}
k & \text{if } k=j^2 \text{ for some } j, \\
0 &\text{otherwise}.
\end{cases}
$$
Explicitly, this gives the following pair of recurrence relations for $\sigma_1(k)$ when $k=2m-1$ and $k=2m$:
\begin{subequations}
\begin{multline}\label{sigma1-odd}
\sigma_1(2m-1)+\big(-2\sigma_1(2m-2)+2\sigma_1(m-1)\big) +2\sigma_2(2m-5)
 +\cdots 
\\
=
\begin{cases}
2m-1 & \text{if } 2m-1 = j^2 \text{ for some } j,\\
0 &\text{otherwise};
\end{cases}
\end{multline}
and
\begin{multline}\label{sigma1-odd}
\sigma_1(2m)-\sigma_1(m)-2 \sigma_1(2m-1) +\big(2\sigma_1(2m-4)-2\sigma_1(m-2)\big)
+\cdots
\\
=
\begin{cases}
- 2m & \text{if } 2m = j^2 \text{ for some } j,\\
0 &\text{otherwise}.
\end{cases}
\end{multline}

\end{subequations}

\end{Example}
\begin{Remark}
Define the {\em theta product} \cite[p.~303]{GR90} as
$$\theta(z;q) := \pqrfac{z, q/z}{\infty}{q}.$$
The approach of Example~\ref{ex:JTP-prod} can be used to find the PFE matrix for $\theta(z/a; q)$ 
and its reciprocal, and thus, for rational products of factors of the form $\theta(z/a; q)$; in other words, for a (multiplicative) elliptic function (see
 Rosengren~\cite[Th.~4.12]{HR2016-lectures}).  
\end{Remark}

In this section, we developed a calculus to obtain  the entries of the PFE matrix for a wide variety of functions. We now consider some applications of these ideas. 

\section{Application 1: Roots of generating functions}\label{sec:applications1}
Let $P(n)$ and $b_k$ be the corresponding sequences that satisfy \eqref{sum-prod}. 
In this section we consider the relationship between the divisibility properties of the 
two sequences $P(n)$ and $b_k$. As a by-product, we obtain a result due to
Heninger, Rains and Sloane \cite{HRS2006}, which they obtained in the context of studying
$n$th roots of theta functions.

We first prove a preliminary proposition about $P(n)$ and $b_k$.
\begin{Proposition} Let $P(n)$ and $b_k$ satisfy \eqref{sum-prod}. Then $P(n)$ is an integer (or a 
rational number) if and only if $b_k$ is an integer (respectively, rational number). 
\end{Proposition}
\begin{proof}
Let $b_k$ be an integer for all $k$. Then the $P(n)$ are all integers too. This follows from
\eqref{gen-partition-function} and the fact that $\rfac{b}{r}/r!$ is an integer if $b$ is an integer and $r$ a non-negative integer. 

Conversely, suppose $P(n)$ is an integer for all $n$. Let $F_k(n)$ be the corresponding frequency function. We show all $b_k$ are integers. For the sake of producing a contradiction, let $m$ be the smallest number such that $b_m$ is not an integer.  Clearly, in view of the matrix equation \eqref{gen-a}, $F_k(m)$ for $k=1, 2, \dots, m-1$ are also integers since $P(k)$ and $b_k$ are integers for
$k=1, 2, \dots, m-1$. Consider the product
$$Q_{m-1}(q)=\prod_{k=1}^{m-1}\big( 1-q^k\big)^{-b_k} = \sum_{n=0}^{\infty} P^\prime(n)q^n.$$
Let $F^{\prime}_k(n)$ be the corresponding frequency function. Clearly, $F^\prime_m(m)=0$, and from \eqref{gen-column-sums} we have
$$mP^\prime(m)=\sum_{k=1}^{m-1} kF^\prime_k(m).$$
From the equation \eqref{gen-a} for $n=m$, we must have $F^\prime_k(m)=F_k(m)$ for $k=1, 2, \dots, m-1$. Thus \eqref{gen-column-sums} implies that
\begin{align*}
mP(m) &=mF_m(m)+\sum_{k=1}^{m-1}kF_k(m) \\
&= mF_m(m)+mP^\prime(m).
\end{align*}
This implies that $F_m(m)$ is an integer, and since $b_mP(0)=F_m(m)$, $b_m$ must be an integer. This is a contradiction to our assumption that $b_m$ is not an integer. 

The proof of the result when $P(n)$ and $b_k$ are rational numbers is straightforward in view
of \eqref{gen-partition-function}. 
\end{proof}

Next, we prove the main theorem of this section concerning divisibility properties shared by $P(n)$ and $b_k$.
\begin{Theorem}\label{th:prime-division}
Let $p$ be a prime and $r>0$ a positive integer such that $p^r\mid P(n)$ for all $n>1$. Then,
\begin{enumerate}[(a)]
\item$p^r\mid b_m$  if $(p,m)=1$; and,
\item $p^{r-1} \mid b_m$ if $p\mid m$.
\end{enumerate}
\end{Theorem}
\begin{Remark} If we take $P(n)$ to be the sequence $(1, 4, 4, 4, \dots)$, then a short calculation shows that $b_1=4, b_2=-6.$ This gives an example where $2^2\mid P(n)$ for all $n>1$; $2\mid b_2$, 
but $2^2\nmid  b_2$.
\end{Remark}
\begin{proof}  
First consider the case $(p,m)=1$. The proof is by induction on $m$. It is easy to see that $b_1=P(1)$, so the result holds for $m=1$.
From 
\eqref{divisor-recurrence-b}, we have
\begin{align*}
mb_m &=mP(m)-\sum_{d\mid m \atop d<m} db_d-\sum_{k=1}^{m-1} \sum_{d\mid k} db_d P(m-k)\\
&= (I) - (II) - (III).
\end{align*}
Now, $(I)$ and all the terms of (III) are divisible by $p^r$ by hypothesis. Further, for the sum in (II), since $d\mid n$, so $(d, p)=1$, thus $p^r\mid b_d$
for all divisors $d$ of $n$ by induction. Thus $p^r \mid mb_m$ which implies $p^r\mid b_m$ since $(p,m)=1$. 

Next suppose $p\mid m$. 
Again $(I)$ and $(III)$ are divisible by $p^r$ by  hypothesis. The middle term can be written as
$$(II)=\sum_{d\mid m \atop d<m, (d,p)=1} d\cdot b_d + \sum_{d\mid m \atop d<m, p\mid d} d\cdot b_d.
$$
The first of these is divisible by $p^r$ by part (a). The second will be divisible by $p^{r-1}$ if for all
divisors  $d$ of $m$ that are smaller than $m$ such that $(p,d)=1$, $b_d$ is divisible by $p^{r-1}$. Repeating this process we see that if we prove the result for $m=p$, then we will be done. But when $m=p$, we have 
$$p\cdot b_p = pP(p) - b_1-\sum_{k=1}^{p-1} \sum_{d\mid k} d\cdot b_d P(p-k).$$
Clearly, each term on the right is divisible by $p^r$, so $p^{r}\mid pb_p$ or $p^{r-1}\mid b_p$. This completes the proof for $m=p$ and for part (b) of the theorem.
\end{proof}

Recall the notations \eqref{def-pr} and \eqref{sum-prod}.
\begin{Corollary}Let $P(n)$ with $P(0)=1$ be as in \eqref{sum-prod}. Suppose for $n>0$, $P(n)$ are  multiples of a fixed integer $m^t$; then the formal series given by the $m^s$th roots of $Q(q)$ will have integer coefficients for all $s<t$.  
\end{Corollary}
\begin{Remark} This is another way to state a result of
Heninger, Rains and Sloane \cite[Cor.\ 2]{HRS2006}. 
\end{Remark}
\begin{proof}
In the notation \eqref{def-pr}, we are interested in the coefficients $P_r(n)$ with $r=1/m^s$. 
Let $b_k$ correspond to the given sequence $P(n)$ in \eqref{sum-prod}. 
By Theorem~\ref{th:prime-division}, $b_k/m^s$ are all integers since $s\le t-1$. This implies the $P_r(n)$ are also integers, where $r=1/m^s$, for all $s< t$. 
\end{proof}
In this section we considered the $n$th roots of generating functions. In the next section, we consider arbitrary powers of generating functions. 

\section{Application 2: Powers of generating functions}\label{sec:application2}
The objective of this section is to obtain a recurrence formula for powers of an arbitrary power series, and note some applications. Previously, Gould~\cite{Gould1974} has given its history and applications in combinatorial settings. This  result contains a formula of Lehmer~\cite{Lehmer1951} (see \eqref{ramanujan-gen} below) concerning powers of the eta function. 

We consider powers of $Q(q)$ and use the notation $P_r(n)$ (for the coefficients of $Q(q)^r$)
 defined by 
\begin{equation}\label{def-pr}
Q(q)^r = \sum_{n=0}^\infty P_r(n)q^n.
\end{equation}
Here $Q(q)^r$, for complex $r$, is considered a formal power series. We take $P_r(0)=1$ and $P_r(-m)=0$, for $m>0$.  
We have the following theorem.
\begin{Proposition}\label{prop:pr-ps-rec} Let $r$ and $s$ be non-zero complex numbers and $P_r(n)$ and $P_s(n)$ be defined by  \eqref{def-pr}. Then we have the recurrence relation:
\begin{equation}\label{pr-ps-recurrence}
\sum_{j=0}^n \left(n-(r/s+1)j\right) P_r(n-j)P_s(j) = 0.
\end{equation}
\end{Proposition}
\begin{Remark} Gould~\cite{Gould1974} mentions a variant of this result for $s=1$. We can obtain \eqref{pr-ps-recurrence} from its $s=1$ special case by replacing $Q(q)$ by  $Q(q)^s$  and $r$ by $r/s$.  Gould attributes the general case to Rothe (1793). Lehmer used $s=-3$. We emphasize that $r$ and $s$ are arbitrary complex numbers. Gould~\cite{Gould1974} has extended the idea to Dirichlet series; see Scott and Sokal~\cite{SS2009-11} for some extensions to multivariables.
\end{Remark}
\begin{proof} By Theorem~\ref{th:series-PFE} the hypothesis of Proposition~\ref{prop:divisor-sums} holds with $z=1$. 
We take $f(k)=k$ and $b_k=r, s$  and use \eqref{divisor-sum-gf1-b} to find that
\begin{align*}
rQ(q)^r \sum_{k=1}^\infty g(k)q^k &= \sum_{n=0}^\infty nP_r(n)q^n, \\
\intertext{and,}
sQ(q)^s \sum_{k=1}^\infty g(k)q^k &= \sum_{n=0}^\infty nP_s(n)q^n,
\end{align*}
where $g(n)$ is as defined in \eqref{gn}.
This gives
$$
\frac{r}{s}  \Big(\sum_{n=0}^\infty P_r(n)q^n\Big)\Big(\sum_{n=0}^\infty nP_s(n)q^n\Big) 
=
\Big(\sum_{n=0}^\infty nP_r(n)q^n\Big)\Big(\sum_{n=0}^\infty P_s(n)q^n\Big).
$$
The recurrence \eqref{pr-ps-recurrence} follows by comparing coefficients of $q^n$ on both sides. 
\end{proof}
Some examples of results following from Proposition~\ref{prop:pr-ps-rec} are as follows. 
\begin{Example}[Convolutions of Fibonacci Numbers] Let $\fib(n)$ denote the Fibonacci numbers. The generating function is given by 
$$Q(q)=\frac{1}{1-q-q^2} =\sum_{n=0}^\infty \fib(n)q^n.$$
We take $s=-1$ and let $\fib_r(n)$ denote the sequence generated by $Q(q)^r$. Then we have
the three-term recursion
$$n\fib_r(n) = (n+r-1)\fib_r(n-1)+(n+2(r-1))\fib_r(n-2).$$
The initial cases $\fib_r(0)=1$, $\fib_r(1)=r$ can be used to compute the sequences.
Special cases of $\fib_r(n)$ for small positive integral values of $r$ have appeared in combinatorial contexts; see  OEIS~\cite[A001628, A001629, A001872-5]{sloane}. At the time of writing this paper, a three-term recurrence does not seem to be known for any of these. Note that $r$ need not be a positive integer in our formula. It can be taken to be a complex number; all we need is that the formal power series $Q(q)^r$ is well defined.
\end{Example}

\begin{Example}[Powers of the $\eta$ function] 
Consider
$$Q(q)=\frac{1}{\pqrfac{q}{\infty}q}.$$
When a $q$-expansion is known for one value of $s$, it can be used to find a recurrence relation for 
$P_r(n)$, where $P_r(n)$ is defined as in \eqref{def-pr}.
Take $s=-1$ in Proposition~\ref{prop:pr-ps-rec} and use \eqref{PNT} to obtain for any $r$
\begin{equation}\label{lehmer-gen}
\sum_{j=-\infty}^\infty (-1)^j\left(n+(r-1)j(3j-1)/2\right) P_r\big(n-j(3j-1)/2\big) =0.
\end{equation}
When $r$ is negative, this gives a recurrence for the coefficients of the powers of the $\eta$ function defined as
$$\eta(q):=q^{1/24}\pqrfac{q}{\infty}{q},$$
where $q=e^{2\pi i \tau}$ for $\Im(\tau)>0$. When $r=-24$, $P_r(n)=\tau(n)$, and we obtain a result used by Lehmer~\cite{Lehmer1943} to compute values of $\tau(n)$. 

Instead of \eqref{PNT}, we use \eqref{jacobi1} and Proposition~\ref{prop:pr-ps-rec} with $s=-3$ to obtain
a result of Lehmer~\cite{Lehmer1951}:
\begin{equation}\label{ramanujan-gen}
\sum_{j=0}^\infty (-1)^j(2j+1)\left(n+(r/3-1)j(j+1)/2\right) P_r\big(n-j(j+1)/2\big) =0.
\end{equation}
This is an extension of Ramanujan's recurrence for $\tau(n)$ highlighted in the introduction \eqref{tau-ramanujan}.
\end{Example}

\begin{Example}[Sums of squares and triangular numbers] 
In the examples so far, we have used theta function identities which are powers of $\pqrfac{q}{\infty}{q}$ to derive recurrence relations. In view of Theorem~\ref{th:series-PFE}, we can apply the same approach to other generating functions. In fact, we don't need the 
product form to obtain such recurrences. We use Ramanujan's notation \cite[p.~7]{Berndt2006}; let
$$\varphi(q) :=\sum_{k=-\infty}^\infty q^{k^2} = 1+2\sum_{k=1}^\infty q^{k^2};
\text{ and, } \psi(q) :=\sum_{k=0}^\infty q^{\frac{k(k+1)}{2}} .$$
Let $r_k(n)$ and $t_k(n)$ be defined as the coefficients of
$$\varphi(q)^k =\sum_{n=0}^\infty r_k(n)q^n
\text{ and }
\psi(q)^k =\sum_{n=0}^\infty t_k(n)q^n.
$$
Then $r_k(n)$ is the number of ways of writing $n$ as an ordered sum of $k$ integers and $t_k(n)$ the number of ways of writing $n$ as an ordered sum of $k$ triangular numbers. (Here order matters: $1+6=6+1$ are considered different.)

We take $s=1$ and $Q(q)=\varphi(q)$ in Proposition~\ref{prop:pr-ps-rec}. Note that
$$r_1(n)=\begin{cases}
1 & \text{if } n =1,\\
2 & \text{if } n =j^2 \text{ for some } j,\\
0 & \text{otherwise}.
\end{cases}
$$
Using this we find an identity which is proved using Liouville's methods in Williams \cite[p.~77]{Williams2011} (see also Venkov~\cite[p.~204]{Venkov1970}):
\begin{equation}\label{squares-rec1}
nr_k(n)+2\sum_{j=1}^\infty \big(n-(k+1)j^2\big) r_k(n-j^2) = 0.
\end{equation}
Similarly, by taking $Q(q)=\psi(q)$, we obtain
\begin{equation}\label{triangular-rec1}
nt_k(n)+\sum_{j=1}^\infty \big(n-(k+1)j(j+1)/2 \big)t_k(n-j(j+1)/2) = 0.
\end{equation}

We have been unable to find this result in the literature; it appears to be new. 
\end{Example}
 Next, we extend \eqref{squares-rec1} by writing a recurrence for the powers of the summation side of Jacobi's triple product identity.
\begin{Example}[Powers of Jacobi's Triple Product identity] \label{JTP-1} 
Let $$J(q) = 1+\sum_{n=1}^\infty\big( z^{n}+z^{-n}\big) q^{n^2},$$
and define $J_r(n,z)$  as the coefficients of
$$J(q)^r = \sum_{n=0}^\infty J_r(n,z)q^n.$$
Applying \eqref{pr-ps-recurrence} with $s=1$ we obtain
\begin{equation}\label{JTPI-rec1}
nJ_r(n,z)+\sum_{j=1}^\infty \big(n-(r+1)j^2 \big)\big(z^{j}+z^{-j}\big)J_r(n-j^2) = 0.
\end{equation}
We emphasize that $r$ need not be a positive integer in this formula. 
\end{Example}



\section{Application 3: Infinite families of congruences}\label{sec:applications3}
In this section we apply the results in the previous section to obtain some results that are in the same vein as Ramanujan's congruences for $p(5m+4)$ and $\tau(5m)$ mentioned in the introduction
and recent work of Chan and Wang~\cite{CW2019}.
The following theorem gives four infinite families of such congruences.

\begin{Theorem}\label{th:congruence5} Let $m$ be a non-negative integer and $r$ a rational number. 
Let $Q(q)=1/\pqrfac{q}{\infty}{q}$, and $P_r(n)$ be defined by \eqref{def-pr}.  Then we have the following infinite families of congruences:
\begin{enumerate}
\item $P_r(5m+1) \equiv 0$ $(\modulo 5)$ if $r \equiv 0$ $(\modulo 5)$;
\item $P_r(5m+2) \equiv 0$ $(\modulo 5)$ if $r \equiv 2$ $(\modulo 5)$;
\item  $P_r(5m+3) \equiv 0$ $(\modulo 5)$ if $r \equiv 4$ $(\modulo 5)$;
\item  $P_r(5m+4) \equiv 0$ $(\modulo 5)$ if $r \equiv 1$ $(\modulo 5)$.
\end{enumerate}
\end{Theorem}
\begin{Remark} The cases $r=1$ and $r=-24$ in (4) give Ramanujan's congruences mentioned in the introduction. See Berndt~\cite[Ch.\ 2]{Berndt2006} for other proofs.
A few congruences given by Chan and Wang~\cite{CW2019} are also included in the above;
a part of the assertions in their equations (3.1), (3.8), (3.11), (3.41), (3.43), (3.48), and (3.49),
are covered by the above. 
\end{Remark}
\begin{proof}
We use an inductive argument using the recurrence relation \eqref{ramanujan-gen} in the form
\begin{equation}\label{ramanujan-gen-2}
nP_r(n)=\sum_{j=1}^\infty (-1)^{j+1}(2j+1)\left(n+(r/3-1)j(j+1)/2\right) P_r\big(n-j(j+1)/2\big).
\end{equation}

First, let $r \equiv 1$ (mod $5$). For $m=0$, \eqref{ramanujan-gen-2} reduces to
$$4P_r(4)=(9+r)P_r(3)-5(r+1)P_r(2)$$
so $P_r(4) \equiv 0$ (mod $5$) if $r\equiv 1$ (mod $5$). For $m>0$, 
consider the general term 
$$(-1)^{j+1}(2j+1)\left(n+(r/3-1)j(j+1)/2\right) P_r\big(n-j(j+1)/2\big)$$
in \eqref{ramanujan-gen-2} for each $j$. 
It is easy to see that when $j \equiv 4, 5$ (mod $5$), this reduces to an expression of the form
$$\big( * * *\big) P_r(5m+4-5k),$$
for some $k$ and so is $\equiv 0$ (mod $5$). 
When $j\equiv 3$ (mod $5$), of the form
$$(-1)^{j+1} 2 (2r-2) P_r(**) \text{ (mod $5$)};$$
when $j\equiv 2$ (mod $5$), then the expression is of the form
$$(-1)^{j+1} 5 (1+r) P_r(**) \text{ (mod $5$)};$$
and, when 
$j\equiv 1$ (mod $5$), of the form
$$(-1)^{j+1}  (9+r) P_r(**) \text{ (mod $5$)}.$$
In each case, we see that when $r \equiv 1$ (mod $5$), the term is congruent to $0$ (mod $5$). Thus each term of the sum on the right-hand side of \eqref{ramanujan-gen-2} is $0$ (mod $5$). This completes the inductive proof of part (4) of the theorem. 

The proof of parts (1)-(3) is similar. For each part, for the given $r$, we consider the general term of the sum as above, and consider the cases where the index $j$ is congruent to $1, 2, 3, 4, 5$ (mod $5$). 
\end{proof}
By using virtually the same argument, we obtain the following congruences modulo $3$.
\begin{Theorem} Let $m$ be a non-negative integer and $r$ an integer. 
Let $P_r(n)$ be as defined in \eqref{def-pr}.  Then if  
$r\equiv 0 \; (\modulo 3)$, we have 
$$P_r(3m+k) \equiv 0 \; (\modulo 3) \text{ for  } k=1, 2.
$$
\end{Theorem}
\begin{proof}  We use \eqref{ramanujan-gen-2} again. The details are very similar to the proof of
Theorem~\ref{th:congruence5}. 
\end{proof}

\section{Concluding Remarks}

About his recurrence \eqref{euler-sigma}, Euler~\cite{LE1760-243} (translated by Jordan Bell) wrote:
\begin{quote}
Since this is the case, I seem to have advanced the science of numbers by not a small amount when I found a certain fixed law according to which the terms of the given series $1, 3, 4, 7, 6,$ 
etc.\
progress, such that by this law each term of the series can be defined from the preceding; for I have found, which seems rather wonderful, that this series belongs to the kind of progression which are usually called recurrent and whose nature is such that each term is determined from the preceding according to some certain rule of relation. And who would have even believed that this series which is so disturbed and which seems to have nothing in common with recurrent series would nevertheless be included in this type of series, and that it would be possible to assign a scale of relation to it?  
\end{quote}
Clearly, Euler's enthusiasm for recurrence relations of this type has been shared by many 
authors---including Ramanujan, Lehmer and Ewell---perhaps because these results provide a 
practical technique to compute values of the relevant functions. As we have seen, recurrences of 
this type can be found by considering the partition-frequency enumeration matrix and its associated constructs. 
In addition, we have 
seen some further applications which suggests that this representation is quite useful in obtaining number-theoretic information in combinatorial contexts. 

Finally, we mention a promising direction for the future. We have not found the moments of the frequency function $F_k(n)$ in the theory of partitions. However, in
Example~\ref{ex:moments}, we saw that the moments $M_m(n)$ are related to the divisor function $\sigma_m(n)$. Note further that the powers of the $\eta$-function are related to the hook lengths of a partition by means of the celebrated Nekrasov-Okounkov formula (see also Han~\cite{Han2010});
 one can use this to obtain formulas for the moments, as is done by Han~\cite[Eq.\ (6.4)]{Han2008}. The analogues of the frequency function and their moments for other functions may turn out to be equally interesting. 

\subsection*{Acknowledgments} We thank Alan Sokal for bringing some references to our attention. Part of this work was done when the second named author was visiting the School of Physical Sciences (SPS), Jawaharlal Nehru University, Delhi. We thank the anonymous referee for many helpful suggestions.


\end{document}